\PassOptionsToPackage{usenames,dvipsnames}{color}
\documentclass[12pt,letterpaper,english,reqno]{amsart}

\usepackage{amsmath,amsfonts,amssymb,mathrsfs,stackrel}
\usepackage[english]{babel}
\usepackage[utf8]{inputenc}
\usepackage[T1]{fontenc}
\usepackage{graphicx}
\usepackage{tabularx}
\usepackage{mathtools}
\usepackage{hyperref}
\usepackage{cleveref}
\usepackage{comment}
\usepackage[all]{xy}
\usepackage{xurl}
\usepackage{soul}
\usepackage{tikz}
\numberwithin{equation}{section}

\hypersetup{colorlinks,citecolor=blue,linktocpage,hyperindex=true,backref=true}

\newcommand\cB{{\mathcal B}}

\newcommand\cH{{\mathcal H}}

\newcommand\cO{{\mathcal O}}

\newcommand\cT{{\mathcal T}}

\newcommand\cX{{\mathcal X}}
\newcommand\cY{{\mathcal Y}}

\newcommand\bA{{\mathbb A}}

\newcommand\bC{{\mathbb C}}
\newcommand\bD{{\mathbb D}}

\newcommand\bG{{\mathbb G}}

\newcommand\bP{{\mathbb P}}
\newcommand\bQ{{\mathbb Q}}
\newcommand\bR{{\mathbb R}}
\newcommand\bS{{\mathbb S}}

\newcommand\bZ{{\mathbb Z}}

\newcommand\wE{{\widetilde E}}

\DeclareMathOperator{\SL}{SL}
\DeclareMathOperator{\GL}{GL}
\DeclareMathOperator{\Sp}{Sp}

\newtheorem{theorem}{Theorem}[section]

\newtheorem{conjecture}[theorem]{Conjecture}
\newtheorem{corollary}[theorem]{Corollary}
\newtheorem{question}[theorem]{Question}

\newtheorem{lemma}[theorem]{Lemma}
\newtheorem{proposition}[theorem]{Proposition}
\theoremstyle{definition}
\newtheorem{construction}[theorem]{Construction}

\newtheorem{definition}[theorem]{Definition}

\newtheorem{remark}[theorem]{Remark}

\title[Complex $\bS^6$]{Complex structures on $\bS^6$}
\author{Philip Engel}
\address{Department of Mathematics, Statistics, and Computer Science,
University of Illinois Chicago, 851 S Morgan St, Chicago, IL 60607, USA}
\email{pengel@uic.edu}
\date{\today}

\begin{document}
\begin{abstract} 
A recent manuscript produced by Claude
under the direction of Levent Alp\"oge constructs a 1-parameter family 
of compact complex threefolds diffeomorphic to $\bS^6$ \cite{AC26}. 
Our goal is to describe the geometric ideas behind the construction,
give a self-contained proof, and remark on possible future research directions. \end{abstract}
\maketitle

\setcounter{tocdepth}{1}
\tableofcontents

\section{Introduction}\label{sec-intro}

\subsection{History} In 1948, Hopf proved
that $\bS^4$ and $\bS^8$ do not admit complex structures
\cite{Hop48} and in 1953, Borel
and Serre \cite{BS53} used topological methods to 
prove that only $\bS^{2}$ and $\bS^6$ admit an almost
complex structure. Since $\bS^2\simeq \bP^1$, this raises the question:

\begin{question}[Hopf Problem]\label{hopf-q}
Does $\bS^6$ admit a complex structure?
\end{question} 

Yau asked whether every compact almost complex manifold
of dimension at least $6$ also admits a complex structure \cite[Prob.~52]{Yau93}.

Both negative and positive answers to Question \ref{hopf-q} have
been proposed, but these works were
found not to be fully mathematically sound. See 
\cite{ABGKR18} and \cite{Bry21} for discussions. 
Thus, the problem acquired a certain degree of notoriety.
Hypothetical properties of such a complex structure, for instance, 
Dolbeault cohomology groups and algebraic dimension,
were studied in \cite{Gra97}, \cite{CDP98}.

In 2026, L.~Alp\"oge\footnote{The precise nature of the 
interaction between Alp\"oge and Claude is not, at present, documented.}
posted a 108-page manuscript \cite{AC26} produced by an artificial intelligence,
Anthropic's ``Claude'', 
claiming to construct a complex 
structure on $\bS^6$. Perhaps the most remarkable aspect
of the construction is that it fibers holomorphically 
over the Riemann sphere $\bP^1$, contradicting
a nearly 30-year-old result in the literature that the algebraic dimension of such a complex
structure must be zero \cite[Cor.~4.2]{CDP98}. See \cite[\S~10]{AC26} for a further discussion.

\subsection{Outline}

We begin with the rational elliptic surface $S$ having
singular fibers of Kodaira types $IV^*$, $III$, $I_1$.  From a line bundle
$M\in {\rm Pic}(S)$, one forms a principal $\bG_m$-bundle over $S$, then quotients it by a
lift of translation along a section of $S\to \bP^1$. 
This linearization acquires a zero over $\infty\in \bP^1$ where one fills the quotient family
by a Mumford construction.
The result is a compact complex
$3$-fold $X(u)$ fibering over $\bP^1$, with complex $2$-torus fibers away from $0,1,\infty$.
It depends on a single complex parameter $u\in \Delta^*$. See Section \ref{sec-initial}. 

Using the potential good reduction of $S$ over $0,1\in \bP^1$ to perform
logarithmic transforms, one can modify the family over $0,1$. The replaced fibers
have multiplicities $3,4$ respectively, and their reductions are bielliptic surfaces. 
The result is a compact complex $3$-fold $Y=Y(u)$ fibering over $\bP^1$.
See Section \ref{sec-log-mod}. 
The monodromy representation of the fibration 
$Y\to \bP^1$ is computed in Section \ref{sec-monodromy}.
Using this, one checks by van Kampen that $Y$ is simply connected (Section \ref{sec-fd}), 
and by the Leray spectral sequence, that $Y$ is a $\bZ$-homology sphere (Section \ref{sec-homology}).
Hence $Y$ is diffeomorphic to $\bS^6$ (Section \ref{sec-main}).

We speculate about possible future research directions in Section \ref{sec-speculate}.

\subsection{Relation to the original manuscript} In \cite{AC26}, the geometry of the construction
is somewhat obscured. The periods of the complex $2$-tori are the central focus, encoded
by a period matrix
$$\Pi(z) = \begin{pmatrix} 6\mu(z) & \tau(z) & 1 & 0 \\ \beta(z) & \mu(z) & 0 & 1 \end{pmatrix}.$$
Here $z$ is a coordinate on the universal cover of $\bP^1\setminus \{0,1,\infty\}$. Roughly,
the relation to this paper is that $\tau(z)$ is the period of the elliptic fibration $S\to \bP^1$,
$\mu(z)$ encodes the third period arising from the circle of the principal $\bG_m$-bundle, 
and $\beta(z)$ encodes the fourth period, from the quotient by the $\bZ$-action of the lifted translation.

This paper recasts several arguments in \cite{AC26}, such as 
the van Kampen and Leray computations. Since our geometric approach differs 
substantially from a period-theoretic presentation, 
we give a self-contained account of the construction and proof.
We have omitted some material from the original text, 
such as a detailed discussion of \cite{CDP98}
and computations of Dolbeault cohomology groups.

\subsection{AI disclosure} The exploration of the original paper and
drafting of arguments and computations were performed in collaboration 
with ChatGPT. The mathematical content was 
independently verified by the author, who takes responsibility for the
manuscript.

\section{Initial construction}\label{sec-initial}
Let
$
\pi:S\to \bP^1
$
be a rational elliptic surface whose $j$-map has degree $1$, so that
in suitable coordinates,
$
j=1728t
$, and for which there
are three
singular fibers, of Kodaira types
$
IV^*$, $III$,  $I_1$ \cite[p.~206]{Mir90}. 
See Figure \ref{ell-fig}.
The singular fibers
correspond respectively to the order $3$, $4$, $\infty$ 
corners of the $(3,4,\infty)$ triangle.
Concretely,
$$
y^2=x^3-3t^3(t-1)x+2t^4(t-1)^2.
$$

There is a zero section $O$, and the Mordell--Weil group is infinite
cyclic \cite[Main Thm., no.~49]{OS91}. Let
$
P\in {\rm MW}(S)
$
be a generator. The section $P$ passes through non-identity
components on both the $IV^*$ and $III$ fibers, whose component
groups are $\bZ/3\bZ$ and $\bZ/2\bZ$ respectively.

\begin{figure}
\includegraphics[width=0.8\textwidth]{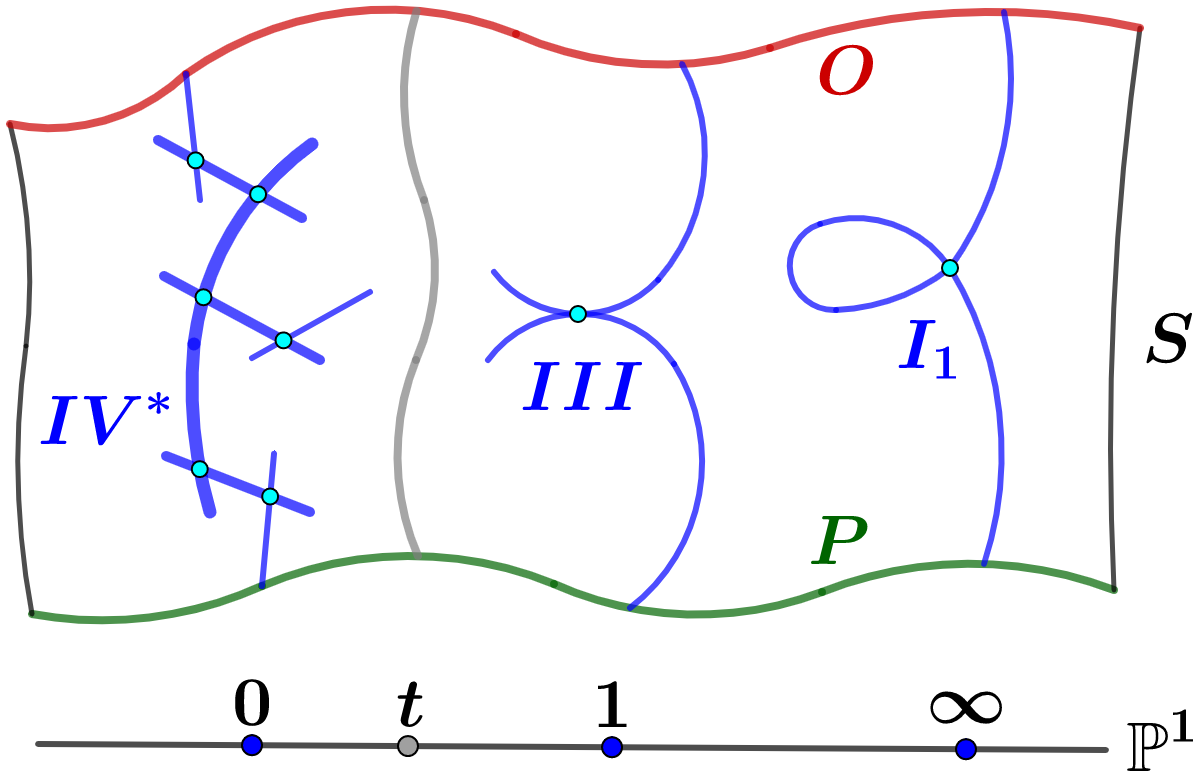}
\caption{Rational elliptic surface $S$ with $IV^*$, $III$, $I_1$ singular fibers
in blue (respectively:~$\wE_6$ configuration, tacnode, and irreducible nodal curve).
Origin section $O$ in red, generator $P$ of the Mordell-Weil group in green.}
\label{ell-fig}
\end{figure}

Consider the line bundle
$
M=\cO_S(P-O).
$
Translation by $6P$ defines an automorphism
$
t_{6P}:S\to S.
$ Since $2$ and $3$ both divide $6$,
the automorphism $t_{6P}$ acts in a manner
preserving the components of the $IV^*$ and $III$
Kodaira fibers.

\begin{lemma}
$\cH om(t_{6P}^*M, M)\simeq  \pi^*\cO_{\bP^1}(1)$.
\end{lemma}

\begin{proof}
For a general fiber $S_t$ over $t\in  \bP^1\setminus \{0,1,\infty\}$, we have
that $M_t\in {\rm Pic}^0(S_t)$. It follows that $t_{6P}^*M_t\simeq M_t$ for all
$t\in \bP^1\setminus \{0,1,\infty\}$. Hence $t_{6P}^*M\otimes M^{-1}$ is expressible as the line bundle 
associated to a linear combination of vertical divisors. Since $t_{6P}$ acts by an element
of ${\rm Aut}^0$ of every fiber, the line bundle $L\coloneqq \cH om(t_{6P}^*M, M)$ has multidegree
$0$ on every fiber. Thus, $L\simeq \pi^*\cO_{\bP^1}(n)$
for some integer $n$, seeing as $\pi$ has no multiple fibers.

We compute $n$ via Shioda's height pairing. Since $S$ is rational
elliptic, we have $\chi(\cO_S)=1$, $O^2=-1$, and $P\cdot O=0$.
The local contributions at the $IV^*$ and $III$ fibers
are $\frac{4}{3}$ and $\frac{1}{2}$, respectively
\cite[Thm.~8.6 and Table~8.9]{Shi90}. Hence
\[
\langle P,P\rangle
=2-\tfrac{4}{3}-\tfrac{1}{2}
=\tfrac{1}{6}.
\]
Consider the section $Q=6P$. By bilinearity,
\[
\langle Q,Q\rangle
=36\langle P,P\rangle
=6.
\]
Since $Q$ meets the identity component of every
reducible fiber, all local contributions involving $Q$ vanish, and
Shioda's height formula gives
\[
6=\langle Q,Q\rangle=2+2(Q\cdot O).
\]
It follows that
$
(6P)\cdot O=Q\cdot O=2.
$
Similarly,
$
\langle P,Q\rangle
=6\langle P,P\rangle
=1
$
and in turn
$
P\cdot(6P)=P\cdot Q=2.
$

Finally, restricting $L=t_{6P}^*M^{-1}\otimes M$ to the zero section gives
\begin{align*}
n
&=L\cdot O\\
&=(O-P)\cdot(6P)+(P-O)\cdot O\\
&=(2-2)+(0-(-1))=1.\qedhere
\end{align*}
\end{proof}

\begin{corollary}\label{lin-exist}
There is, up to scale,
a unique section $$\phi\in \cH om( t_{6P}^*M, M)(S)$$
which over $\bA^1$ is an isomorphism, i.e.~a linearization of the action
of $t_{6P}$ on $M$, and has a simple zero along the fiber
over $\infty \in \bP^1$.
\end{corollary}

Let $M^\times \coloneqq M\setminus \{\textrm{zero section}\}.$
It is a threefold, which is a principal $\bG_m$-bundle over 
the rational elliptic surface $S$. Let $S^\circ\coloneqq S\setminus S_\infty$
and $M^{\times, \circ}\coloneqq M^\times \vert_{S^\circ}$.

\begin{definition}\label{z-action}
We define a $\bZ$-action on $M^{\times,\circ}$ by acting by
the translation $(t_{6P})^{-1}$ on the base $S^\circ$ and the linearizing
isomorphism $\phi$ on $M^{\times,\circ}$. That is, the generator of $\bZ$
acts by the lift of $t_{6P}^{-1}$ defined by $\phi$.
\end{definition}

So in particular,
the automorphism of $M^{\times,\circ}$ induced by $\phi$ lifts
the automorphism $t_{6P}^{-1}$ of $S^\circ$.

\begin{proposition}\label{discontinuous}
 Fix a reference linearization $\phi$ as in Corollary
\ref{lin-exist}. There exists a constant
$c>0$ such that for all linearizations $u\phi$, $u\in \bC^*$, with $|u|<c$, the $\bZ$-action 
(Definition \ref{z-action}) associated to $u\phi$ is free and properly discontinuous. \end{proposition}

\begin{proof}
Consider the map
$
\phi\colon t_{6P}^*M\to  M
$ which is an isomorphism away from
$S_\infty$.
Choose a smooth Hermitian metric on $M$. With respect to this metric,
the operator norm
$
x\mapsto \|\phi_x\|
$
is a continuous function on the compact surface $S$, equal to zero 
exactly along $S_\infty$. It is therefore bounded. Set
$
C=\max_{x\in S}\|\phi_x\|
$
and choose $c>0$ such that $cC<1$.

Fix $u\in\bC^*$ with $|u|<c$. Then,
$u\phi$ also defines a self-map of $M$
covering $(t_{6P})^{-1}$ as in Definition~\ref{z-action}. 
Since the operator norm of $u\phi_x$ is strictly
less than $1$, the action is contracting 
with respect to the Hermitian norm function $M^{\times,\circ} \to \bR_{>0}.$
The $\bZ$-action on this open set is therefore free and properly discontinuous.
\end{proof}

\begin{definition}\label{linearization-def}
Let $\phi_0$ be a linearization as in Corollary \ref{lin-exist} and let
$\phi = u\phi_0$ with $|u|<c$, $u\in \bC^*$ as in Proposition \ref{discontinuous}.
We define a smooth threefold as the quotient 
$$X^\circ(u)\coloneqq M^{\times,\circ}/\bZ$$
of the $\bZ$-action of Definition \ref{z-action} associated 
to the linearization $\phi$.
\end{definition}

\begin{proposition} $X^\circ(u)\to \bA^1$ is a fibration with
smooth complex $2$-torus fibers over $\bA^1\setminus \{0,1\}$.
 \end{proposition}
 
 \begin{proof} Over $t\in \bA^1\setminus \{0,1\}$ we have an exact sequence
 \begin{align}\label{exact} 1\to \bC^*\to M^\times_t \to S_t\to 0,\end{align}
 i.e.~the fiber of $M^\times$ over $t$ is a semiabelian variety.
 So the universal cover of $M_t^\times$
 is $\bC^2$, and its fundamental group is isomorphic to $\bZ^3$.
 
 Quotienting by the additional
 $\bZ$-action associated to the linearization $\phi$ gives a complex $2$-torus,
 because the action is properly discontinuous, see Proposition \ref{discontinuous},
 and the additional $\bZ$-factor increases the rank of the subgroup of $\bC^2$ from $3$ to $4$.
 \end{proof}
 
\begin{proposition}\label{mumford-1}
In a punctured neighborhood of $\infty\in \bP^1$ with
coordinate $q$,
$X^\circ(u)_q$ is biholomorphic to $$(\bC^*)^2/\langle (c_1q, c_2), 
(c_3, c_4q) \rangle$$ for holomorphic units $c_i\in \bC^*+O(q)$ on $\Delta$.
\end{proposition}

\begin{proof}
Since $S_\infty$ is an $I_1$ Kodaira fiber, the restriction of $S$ to
$\Delta^*$ admits a Tate uniformization
$
S_q\simeq\bC^*/q^\bZ.
$
After pulling $M^\times$ back to the partial cover
$\bC^*\to S_q$ and choosing a trivialization, its total space is
$(\bC^*)^2$. The deck transformation defining the Tate curve acts by
$$
(z,w)\mapsto
(c_1q z,c_2w), \textrm{ for }(z,w)\in (\bC^*)^2.
$$
Here $c_1$ and $c_2$ are holomorphic units. The first scaling factor has
order one in $q$ because it is the Tate period, while the second has
order zero because $M$ has degree zero and extends across the $I_1$
fiber.

The linearization $\phi$ gives a second transformation of
$(\bC^*)^2$. It covers translation by $-6P$ on the Tate curve. Since
$6P$ specializes to a point of the smooth locus of the $I_1$ fiber,
its multiplicative Tate coordinate is a holomorphic unit. Thus the
first component of this transformation has the form $z\mapsto c_3z$.
Moreover, $\phi$ vanishes to first order along $S_\infty$ so its
action in the fiber direction of $M^\times$ has the form
$w\mapsto c_4 q w$, where $c_4$ is a holomorphic unit. Hence the second
transformation is
$$
(z,w)\mapsto
(c_3z,c_4 qw), \textrm{ for }(z,w)\in (\bC^*)^2.
$$

Quotienting $(\bC^*)^2$ by these two commuting transformations gives
the asserted description of $X^\circ(u)_q$ for $q\in \Delta^*$.\end{proof}

\begin{proposition}\label{mum-cor}
A $\bZ^2$-periodic tiling 
of $\bR^2$ by lattice simplices of area $\tfrac{1}{2}$ determines
a Kulikov ($K$-trivial, semistable)
filling $$X^\circ(u)\hookrightarrow X(u)$$
over $\infty\in \bP^1$.
For example, one may take the tiling by triangles in Figure \ref{kulikov-fig}.
In this case, the fiber
$X(u)_\infty$ is the result of gluing the opposite sides of a hexagon of $(-1)$-curves
on a del Pezzo surface $dP_6$.\end{proposition}

\begin{proof} This follows from a simple Mumford fan construction
\cite{Mum72}. By Proposition \ref{mumford-1}, the
fiber over the punctured disk is $(\bC^*)^2$ modulo the subgroup
generated by the rows of the matrix
$$
c\cdot q^{\begin{psmallmatrix}1&0\\0&1\end{psmallmatrix}},
\qquad
c={\small \begin{pmatrix}c_1&c_2\\c_3&c_4\end{pmatrix}} ,
$$
where $\cdot$ denotes entrywise multiplication.
Thus, a $\bZ(1,0)\oplus \bZ(0,1)\simeq \bZ^2$-periodic
tiling of $\bR^2$ provides a Mumford construction filling the family over $q=0$.
Take the triangulation
of the unit square into two triangles. Extending
$\bZ^2$-periodically gives such a tiling
 $\cT$. Put $\cT$ at height $1$ in $\bR^3$, i.e.~$\bR^2\times \{1\}\subset \bR^2\times \bR\simeq \bR^3$,
 and take the cone ${\rm Cone}(\cT)$. The result is a $\bZ^2$-periodic fan in $\bR^3$
 with standard affine cones. 
 
 This fan defines a one-parameter Mumford construction of degenerating
 complex $2$-tori, as in  \cite[Constr.~3.3 and Rem.~3.5]{EGFS25}:~we take the $\bZ^2$-quotient 
 of the toric variety of ${\rm Cone}(\cT)$, for which the fiber $(\bC^*)^2$ over $q\in \Delta^*$ of the natural map to $\bA^1$ is
 quotiented by the $\bZ^2$-action of Proposition \ref{mumford-1}. 
 
Before quotienting, the central fiber is an infinite periodic quilt of smooth toric surfaces, whose components
are isomorphic to $dP_6$ (see Figure \ref{kulikov-fig}). The passage to the $\bZ^2$-quotient has the effect
 of taking a single such $dP_6$ and gluing the opposite sides
 of its anticanonical hexagon, as described in the proposition.\end{proof}
 
 \begin{figure}
\includegraphics[width=\textwidth]{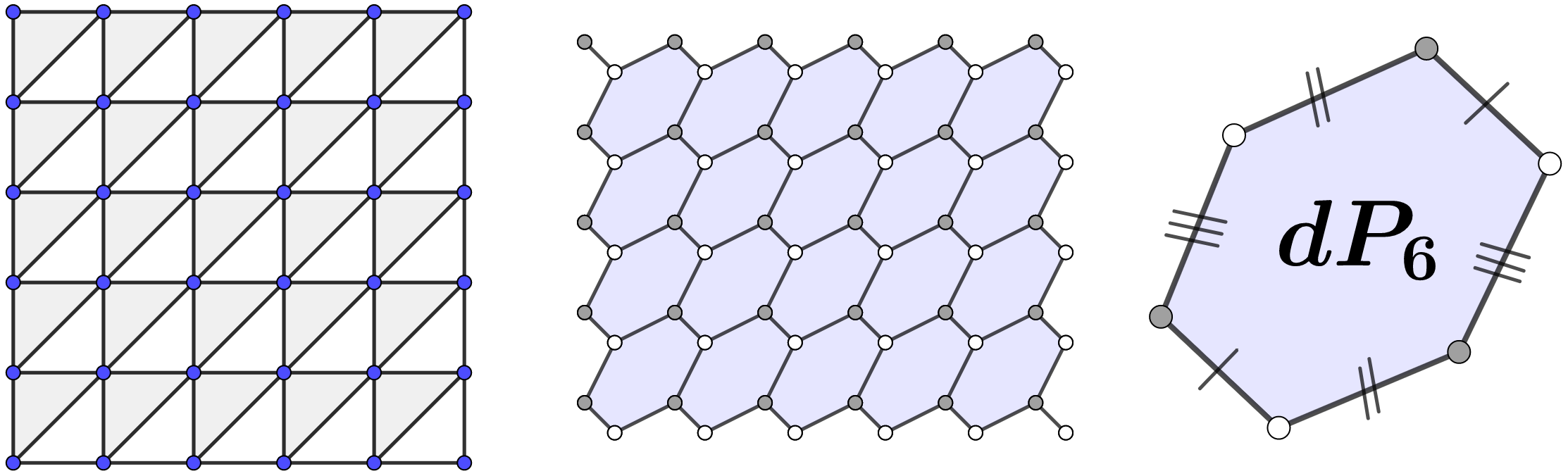}
\caption{Left:~$\bZ^2$-periodic tiling $\cT$ of the plane $\bR^2$. 
Middle:~$\bZ^2$-periodic infinite quilt of smooth toric surfaces.
Right:~Its $\bZ^2$-quotient and the central fiber of Mumford construction, a glued
$dP_6$. Two $0$-dimensional
torus orbits in white and grey, three $1$-dimensional torus orbits in black; the $2$-dimensional
torus orbit in blue.}
\label{kulikov-fig}
\end{figure}

\begin{proposition}\label{section-of-J}
The fibration $X(u)\to\bP^1$ admits a holomorphic section.  Over
$\bP^1\setminus\{0,1,\infty\}$, this section may be taken as the origin of
the smooth complex $2$-torus fibers.
\end{proposition}

\begin{proof}
The restriction of $M$ to the zero section $O$ admits a section $e$
which is non-vanishing over $\bA^1$ and has a first order zero at $\infty$,
because $M\cdot O=1$. Hence 
$e$ defines a section of $M\vert_O\to O$, valued
in $M^\times\vert_O$ away from $\infty$. Its image under
the quotient map of Definition \ref{linearization-def} defines
a section of $X^\circ(u) =M^{\times,\circ}/\bZ \to \bA^1$.

Near $\infty$, the condition that $e$ vanishes to first order
implies that it defines a local section of the
uniformization of Proposition~\ref{mumford-1}. Indeed, we may first
translate $e$ by the period $(c_3^{-1}, c_4^{-1}q^{-1})$. 
This translated section has nonzero limit at $q=0$, and so extends 
as a section of the Mumford construction. It coincides with 
$e$ in the quotient $X^\circ(u)$.
\end{proof}

We have thus completed the summary of the ``initial construction'': We have produced
a smooth, compact complex threefold $X(u)$ for all sufficiently small $u\in \bC^*$,
which is a complex $2$-torus fibration, with section, and three singular fibers.

\section{Logarithmic modifications}\label{sec-log-mod}

We now fill the family $X(u)\to \bP^1$ at $0$ and $1$ in a new way.
This logarithmic transform changes the fiber without
changing the torus family over the punctured neighborhoods of $0$ and $1$, see 
e.g.~\cite[Ch.~V, \S 13]{BHPV04}, \cite[\S 5.3]{EFGMS25}.

To this end, we first describe the family $X(u)$ over a neighborhood of $0$ and $1$.
Observe that the $\SL_2(\bZ)$ monodromies of $S\to \bP^1$ about $0$ and $1$ have orders
$3$ and $4$, respectively.

\begin{proposition}\label{reduction} 
Let $\Delta_0\ni 0$, $\Delta_1\ni 1$ be disks. Consider the order $3$, resp.~$4$
 base change $\Delta\to \Delta_0$, $s\mapsto s^3=t$ or $\Delta\to \Delta_1$, $s\mapsto s^4=t-1$.
\begin{enumerate}
\item For $i=0,1$, $S\times_{\Delta_i}\Delta$ admits a smooth elliptic reduction
$S'\to\Delta$, whose central fiber is isomorphic to $E_\triangle$ for $i=0$
and to $E_\square$ for $i=1$. \smallskip
\item The linearization, viewed as a biholomorphism $\phi\colon M^{\times,\circ}\to M^{\times, \circ}$,
lifts to a biholomorphism $\phi'$ of $M'\coloneqq \cO_{S'}(P'-O')$ where $P'$ and $O'$ are the transforms
of $P$ and $O$.

Furthermore, the action of $\phi'$ on $E_\triangle$ or $E_\square$ is trivial,
and $\phi'$ acts on the $\bG_m$-torsor $M'^\times$ (which over $E_\triangle$ or $E_\square$
corresponds to either a $3$-torsion or a $2$-torsion line bundle $L^\times$, 
respectively) by a scalar in $\bC^*$.
\end{enumerate}
\end{proposition}

\begin{proof}
The smooth reduction of $IV^*$ and $III$ Kodaira fibers after a monodromy-trivializing
base change is a well-known computation; see \cite[Ch.~V, \S 10]{BHPV04}. 
Let $S''$ be the resolution of indeterminacy of
$S\times_{\Delta_i} \Delta\dashrightarrow S'$. The biholomorphism $\phi$ pulls back 
to a biholomorphism $\phi''$ of the $\bG_m$-torsor $M''^{\times}$ over $S''$ corresponding to the
pullback of $M$. 

This lift descends to a linearization $\phi'$ of $M'^{\times}$ by Hartogs' extension---indeed,
the bundles $M''$ and the pullback of $M'$ agree away from the contracted curves of $S''\to S'$. 
So $\phi''$ descends to a big open set of $S'$ and then extends to a linearization of $M'$ over all of $S'$. The biholomorphism
$\phi'$ is equivariant
over the action of $t_{6P'}^{-1}$ on $S'$. By an explicit computation of $S''$ and the smooth model $S'$,
one checks that $P'$ meets a $3$-torsion point on $E_\triangle$ and 
a $2$-torsion point on $E_\square$. Hence $t_{6P'}$ acts trivially on either central fiber.
The corresponding restriction $L=M'\vert_{E_\triangle}$ or $M'\vert_{E_\square}$ is a
$3$-torsion or $2$-torsion element of ${\rm Pic}^0(E_\triangle)$ or
${\rm Pic}^0(E_\square)$, respectively.

It follows that on the central fiber, $\phi'$ is an automorphism of the $\bG_m$-torsor
$L^\times$ (acting trivially on the base and by a scalar on the fiber).
\end{proof}

\begin{corollary}\label{fill} 
The quotient $L^\times/\bZ$ arising from Proposition \ref{reduction} is
an abelian surface
$(\wE_\triangle\times E'_0)/\mu_3$ or
$(\wE_\square\times E'_1)/\mu_2$. \end{corollary}

\begin{proof}
As $L$ is a $3$- or $2$-torsion bundle over $E_\triangle$ or
$E_\square$, its pullback to an \'etale $\mu_3$- or $\mu_2$-cover
$$\wE_\triangle\to E_\triangle\textrm { or }
\wE_\square\to E_\square$$ 
is trivial. Along this pullback, the $\bZ$-action lifts to
an action on $\bC^*$, whose quotient is an elliptic curve $E'_i$. The deck group of
the \'etale cover acts by translation on both the vertical and horizontal factors.
So the quotient is an abelian surface (rather than bielliptic).
\end{proof}

Abstractly, $\wE_\triangle\simeq E_\triangle$ and
$\wE_\square\simeq E_\square$.  We retain the tildes because the covering
maps are nontrivial isogenies of degrees $3$ and $2$, respectively.

\begin{corollary}\label{fill2}
After base changes of orders $3$ and $4$ locally near
$0$ and $1$, respectively, the restriction of $X(u)\to\bP^1$ to the
punctured disk admits a smooth proper extension
$X'\to\Delta$ by complex $2$-tori, with central fiber
as in Corollary \ref{fill}.
\end{corollary}

Let $X(u)_t$ be a smooth complex $2$-torus fiber over a base point
$t\in \bP^1\setminus \{0,1,\infty\}$, and put
$\Lambda\coloneqq H_1(X(u)_t, \bZ)$. Let $\gamma_0, \gamma_1,
\gamma_\infty\in \pi_1(\bP^1\setminus \{0,1,\infty\}, t)$ be standard
counterclockwise based meridians, ordered so that
$\gamma_0\gamma_1\gamma_\infty=1$.  Write the monodromy representation as
$$T_i\coloneqq {\rm Mon}(\gamma_i)\in \SL(\Lambda)\simeq \SL_4(\bZ).$$

Let us comment a bit more on the structure of $\Lambda$ and the rank $4$, weight $-1$ VHS
on the corresponding local system. By construction, the complex $2$-tori over a general point
of $\bP^1$ are $\bZ$-quotients of the semiabelian variety $M^\times_t$, which fits into the exact sequence
(\ref{exact}). Thus we have a natural filtration: \begin{align}\label{filt1} 0\to \Lambda^\times \to \Lambda \xrightarrow{\psi} \bZ\to 0 \end{align}
where $\Lambda^\times = \pi_1(M_t^\times)$ and a further exact sequence  \begin{align}\label{filt2} 0\to \bZ \to \Lambda^\times \to \Lambda'\to 0 \end{align}
where $\Lambda'$ is the rank $2$, weight $-1$ local system underlying the $\bZ$-VHS
of the elliptic fibration $S\to \bP^1$.  After good reduction at $0$ or
$1$, it specializes to $H_1(E_\triangle,\bZ)$ or
$H_1(E_\square,\bZ)$.

\begin{proposition}\label{conj-rat} We have conjugacies
$$T_0\sim_\bQ \small{
\begin{pmatrix}
0&-1&0&0\\
1&-1&0&0\\
0&0&1&0\\
0&0&0&1
\end{pmatrix}},\, T_1\sim_\bQ
\small{\begin{pmatrix}
0&1&0&0\\
-1&0&0&0\\
0&0&1&0\\
0&0&0&1
\end{pmatrix}},\,
T_\infty\sim_\bZ\small{ \begin{pmatrix}
1 & 0 & 1 & 0 \\
0 & 1 & 0 & 1 \\
0 & 0 & 1 & 0 \\
0 & 0 & 0 & 1 
\end{pmatrix}}.$$ 
\end{proposition}

\begin{proof}
After the cyclic base changes of orders $3$ and $4$, the family has good
reduction at $0$ and $1$.  Let $F_0$ and $F_1$ denote the resulting smooth
central fibers.  The isogenies in Corollary~\ref{fill} give
decompositions
\begin{align}\label{over}
\Lambda\simeq 
H_1(F_i,\bZ)\supset
H_1(\wE_{\triangle\textrm{ or }\square},\bZ)\oplus H_1(E'_i,\bZ).
\end{align}
Put $\rho=e^{2\pi i/3}$.  With the complex orientations, the counterclockwise
monodromy acts on the first summand by multiplication by $\rho$
in the basis $(1,\rho)$ at $0$ (type $IV^*$), and by multiplication by $-i$ in the
basis $(1,i)$ at $1$ (type $III$).  It acts trivially on the second summand.  Hence
\[
T_0\sim_{\bQ}\small{
\begin{pmatrix}
0&-1&0&0\\
1&-1&0&0\\
0&0&1&0\\
0&0&0&1
\end{pmatrix}}
\textrm{ and }
T_1\sim_{\bQ}\small{
\begin{pmatrix}
0&1&0&0\\
-1&0&0&0\\
0&0&1&0\\
0&0&0&1
\end{pmatrix}}
\] with the $\bZ$-conjugacy class determined in Section \ref{sec-monodromy}
from the explicit overlattice $\Lambda$.
At $\infty$, Proposition~\ref{mumford-1} identifies the family with the
quotient of $(\bC^*)^2$ by the Tate period and the linearization period,
having respective $q$-orders $(1,0)$ and $(0,1)$. So
\[
T_\infty\sim_{\bZ}\small{
\begin{pmatrix}
1&0&1&0\\
0&1&0&1\\
0&0&1&0\\
0&0&0&1
\end{pmatrix}}.
\]
We determine the monodromy representation exactly in Section \ref{sec-monodromy}.
\end{proof}

Denote the {\it local monodromy invariants} by $$\Lambda^{T_i} = 
\{\lambda\in \Lambda\,:\,T_i(\lambda)=\lambda\}.$$
The group $\tfrac{1}{3}\Lambda^{T_0}/\Lambda^{T_0}$ consists of $3$-torsion
translations on the smooth model $X'\to \Delta$ over the order 
$3$ base change $\Delta\to \Delta_0$
of Corollary \ref{fill2}, 
which are invariant under the $\mu_3$ deck action.
Similarly, the finite group $\tfrac{1}{4}\Lambda^{T_1}/\Lambda^{T_1}$ consists
of the $4$-torsion translations on the smooth model over $\Delta\to \Delta_1$
which are invariant under the $\mu_4$ deck action. These torsion sections
play an important role in the following construction.

\begin{construction}[Log transform]\label{con-log}
Given $\overline{a}_0\in \tfrac{1}{3}
\Lambda^{T_0}/\Lambda^{T_0}$ we may define a twisted $\mu_3$-action $\widetilde{\rho}$ 
on the smooth replacement $X'\to \Delta$ (Corollary \ref{fill2}) of 
the base change $X(u)\times_{\Delta_0} \Delta$ by setting $$\widetilde{\rho}(x) = \rho(x) + 
\overline{a}_0.$$ Here $\rho$ denotes the deck action induced on the chosen
smooth extension $X'$. Over the punctured disk, its
quotient identifies with $X(u)|_{\Delta_0^*}$.
We may then form the twisted quotient
$$X'/\langle \widetilde{\rho}\rangle.$$ 
We require the following lemma: 

\begin{lemma}\label{twist} 
The twisted quotient $X'/\langle \widetilde{\rho}\rangle$ is biholomorphic
over the punctured neighborhood $\Delta_0^*$ to $X(u)\vert_{\Delta_0^*}$. Furthermore,
a lift of $\overline{a}_0$ to some $a_0\in \tfrac{1}{3}\Lambda^{T_0}$ determines
explicitly such a biholomorphism.
\end{lemma}

\begin{proof}
Put $v_0=3a_0\in\Lambda^{T_0}$ and choose the generator of the deck group
so that it acts on the base by $s\mapsto\zeta_3s$.  Over $\Delta^*$ define
the section
$$
\sigma_{a_0}(s)=\frac{\log s}{2\pi i}v_0
=\frac{3\log s}{2\pi i}a_0.
$$
Although $\log s$ is multivalued, moving to a new branch changes
$\sigma_{a_0}$ by an integral multiple of $v_0\in\Lambda$.  It therefore
defines a single-valued holomorphic section of the family of complex tori
over $\Delta^*$.  Let
$$
\Psi_{a_0}(x,s)=(x+\sigma_{a_0}(s),s)
$$
be translation by this section.

Since $v_0$ is fixed by $T_0$, analytic continuation gives
$$
\sigma_{a_0}(\zeta_3 s)
=T_0^{-1}\sigma_{a_0}(s)+a_0.
$$
Consequently,
$$
\Psi_{a_0}\circ\rho
=\widetilde\rho\circ\Psi_{a_0}.
$$
Thus $\Psi_{a_0}$ descends to a biholomorphism between the quotients by
$\rho$ and $\widetilde\rho$ over $\Delta_0^*$.  The quotient by $\rho$ is
$X(u)\vert_{\Delta_0^*}$ which proves the lemma. \end{proof}

There is an exactly analogous construction given an element 
$\overline{a}_1\in \tfrac{1}{4}\Lambda^{T_1}/\Lambda^{T_1}$ 
and a lift $a_1\in \tfrac{1}{4}\Lambda^{T_1}$. The $a_i$ encode the ``multiplicity
class'' of \cite[Def.~5.22]{EFGMS25}.

\begin{definition} We define the analytic space $X(u,a_0,a_1)$ by gluing to
$X(u)\vert_{\bP^1\setminus \{0,1\}}$ the twisted quotients of Lemma~\ref{twist}
associated to $a_0\in \tfrac{1}{3}\Lambda^{T_0}$ and
$a_1\in \tfrac{1}{4}\Lambda^{T_1}$. \end{definition}
\end{construction}

\begin{lemma}\label{covect}
The space of globally monodromy-invariant linear functionals $\Lambda\to \bZ$ 
has rank $1$, generated by the map
$\psi\colon \Lambda\to \bZ$ of (\ref{filt1}). \end{lemma}

\begin{proof} This is immediate from the monodromy
matrices $[T_0]_{\cB_0}$, $[T_1]_{\cB_1}$, and change of basis $C$
computed in Theorem \ref{matrices}.
\end{proof}

We will see the significance of this global invariant covector in the course of the proof of the main theorem.

\begin{corollary}\label{biell}
When $a_0$ and $a_1$ satisfy
$\psi(a_0)\in \tfrac{1}{3}\bZ\setminus \bZ$, $\psi(a_1)\in \tfrac{1}{4}\bZ\setminus \tfrac{1}{2}\bZ$,
the twisted action of $\widetilde{\rho}$ on $X'_i$, for $i=0,1$, is free. Thus
$X(u,a_0,a_1)\to \bP^1$ has the following multiple fibers:
\begin{enumerate} 
\item multiplicity $3$ over $0$, with reduced fiber a bielliptic surface of
Bagnera--de Franchis type $\mu_3\times \mu_3$.\smallskip
\item multiplicity $4$ over $1$, with reduced fiber a bielliptic surface of
Bagnera--de Franchis type $\mu_2\times \mu_4$.
\end{enumerate}
 \end{corollary}

See \cite[Ch.~V, \S 5]{BHPV04} for the classification of bielliptic surfaces.
 
 \begin{proof}
This follows from the description of the abelian surfaces
$F_0$ and $F_1$ in Corollary \ref{fill} and 
of the twisted action in Construction \ref{con-log}.
The covector $\psi$ induces a map $F_i\to \bR/\bZ$, 
on which the twisted action is translation by $\psi(a_i)$. By hypothesis, 
this translation has exact order $3$, respectively $4$. 
Thus every nonidentity element acts without fixed points, so the action is free.
 \end{proof}

\section{Main result}\label{sec-main}

\begin{theorem}\label{main-thm}
Let $Y \coloneqq X(u,a_0,a_1)$ where local monodromy invariants
$a_0\in  \tfrac{1}{3}\Lambda^{T_0}$ and $a_1\in \tfrac{1}{4}\Lambda^{T_1}$ 
are chosen so that $|12\psi(a_0+a_1)|=1$.
 Then for all $u\in \Delta^*$ in a sufficiently small punctured disk,
 $$Y\simeq_{\rm diffeo} \bS^6$$ is a compact complex manifold,
  diffeomorphic to the $6$-sphere.
\end{theorem}

The hypothesis $|12\psi(a_0+a_1)|=1$ can be satisfied. For instance, 
take $a_0=\tfrac{1}{3}e_4$ and $a_1=-\tfrac{1}{4}f_4$ in the notation
of Theorem \ref{matrices}. Furthermore, this condition ensures
that the hypotheses of Corollary \ref{biell} hold.

\begin{proof}
By Section~\ref{sec-fd}, $Y$ is simply connected.  By
Section~\ref{sec-homology}, it has the integral homology of $\bS^6$.
The Hurewicz theorem implies that $Y$ is a homotopy $6$-sphere.
Smale's generalized Poincar\'e theorem identifies it topologically with
$\bS^6$ \cite{Sma61}. 

By the $h$-cobordism theorem, oriented diffeomorphism
classes of homotopy $6$-spheres are classified by the group $\Theta_6$
\cite{Mil65}. This group is trivial by the computation of Kervaire--Milnor
\cite[\S 7]{KM63}. Hence
$Y$ is also diffeomorphic to the standard $6$-sphere.
\end{proof}

\section{The monodromy representation}\label{sec-monodromy}

The overlattice $\Lambda$ in (\ref{over})
is explicitly computable, with the containment in (\ref{over})
of index $3$ and $2$ respectively.
We now compute this overlattice,
record integral normal forms of $T_0$ and $T_1$ 
in an appropriate oriented basis of $\Lambda$, and determine
the change of basis between the bases expressing
these two normal forms.

These computations are necessary to
compute the fundamental group
and integer homology of $Y$. Recall $\rho=e^{2\pi i/3}$.

\begin{theorem}\label{matrices}
Let
$$
e_1,e_2,\delta_0,e_4
=(1,0),(\rho,0),(0,1),(0,\tau(E'_0))\in \bC^2
$$
be a $\bZ$-basis of
$H_1(\wE_\triangle,\bZ)\oplus H_1(E'_0,\bZ)$, and let
$$
f_1,f_2,\delta_1,f_4
=(1,0),(i,0),(0,1),(0,\tau(E'_1))\in \bC^2
$$
be a $\bZ$-basis of
$H_1(\wE_\square,\bZ)\oplus H_1(E'_1,\bZ)$.
The lattices of the two good reduction fibers are, respectively, the
overlattices obtained by adjoining
$$
e_3=\tfrac{1}{3}(2+\rho,1),
\qquad
f_3=\tfrac{1}{2}(1+i,1).
$$
In the local bases
$\cB_0=(e_1,e_2,e_3,e_4)$ and $\cB_1=(f_1,f_2,f_3,f_4)$
of $\Lambda$,
$$
[T_0]_{\mathcal B_0}=\small{
\begin{pmatrix}
0&-1&-1&0\\
1&-1&0&0\\
0&0&1&0\\
0&0&0&1
\end{pmatrix}},
\qquad
[T_1]_{\mathcal B_1}=\small{
\begin{pmatrix}
0&1&0&0\\
-1&0&-1&0\\
0&0&1&0\\
0&0&0&1
\end{pmatrix}}.
$$
The circle classes of the principal $\bG_m$-bundle $M^\times$
are
$$
\delta_0=-2e_1-e_2+3e_3,
\qquad
\delta_1=-f_1-f_2+2f_3.
$$
Furthermore, the columns of the matrix
$$
C=\small{
\begin{pmatrix}
1&1&0&1\\
0&1&0&0\\
0&-1&1&0\\
0&0&0&1
\end{pmatrix}}
$$
express $\mathcal B_1$ in the basis $\mathcal B_0$, and
$C(\delta_1)=\delta_0$.
\end{theorem}

\begin{proof} The adjoining vectors $e_3$ and $f_3$
are determined by the requirement that their
coordinates in $\wE_\triangle$ or $\wE_\square$
are non-zero and fixed by the order $3$, resp.~$4$, automorphism, and
that $\delta_0$ and $\delta_1$ represent the circle classes
of the principal $\bG_m$-bundle $M^\times$.

Multiplication by $\rho$ sends $e_1\mapsto e_2$ and 
$e_2\mapsto -e_1-e_2$, while fixing $\delta_0$ and $e_4$.  Since
$\delta_0=3e_3-2e_1-e_2$, it follows that
$$
T_0e_1=e_2,\qquad T_0e_2=-e_1-e_2,\qquad
T_0e_3=e_3-e_1,\qquad T_0e_4=e_4.
$$
This gives the displayed matrix for $T_0$.  Similarly, multiplication by
$-i$ sends $f_1\mapsto -f_2$ and $f_2\mapsto f_1$.  Using
$\delta_1=2f_3-f_1-f_2$, we obtain
$$
T_1f_1=-f_2,\qquad T_1f_2=f_1,\qquad
T_1f_3=f_3-f_2,\qquad T_1f_4=f_4,
$$
which gives the displayed matrix for $T_1$.

It remains to compare the two local markings in which we have
the displayed normal forms.  In
$\Lambda'=\Lambda^\times/\bZ\delta_i$, the monodromy representation
of the underlying elliptic surface $S\to \bP^1$ gives the identities
$$
[f_1]=[e_1],\qquad
[f_2]=\frac{[e_1]+2[e_2]}{3}
$$
in $\Lambda'\otimes\bQ$ (rational coefficients arise because
 $[e_1],[e_2]$ and $[f_1],[f_2]$
generate sublattices of indices $3$ and $2$ in the homology lattices of
the smooth reductions $E_\triangle$ and $E_\square$ over $0$ and $1$).
Since
$3[e_3]=2[e_1]+[e_2]$ in $\Lambda'$, the second equality is equivalently
\begin{align}\label{quotient-equality}
[f_1]=[e_1],\qquad [f_2]=[e_1+e_2-e_3].
\end{align}
It remains to lift this identification through (\ref{filt2}) and
(\ref{filt1}). 

The restriction of $C$ to $\Lambda^\times$ is determined up to shears
which act trivially on $\bZ\delta_0$ and on $\Lambda'$.  Thus there are
integers $a,c$ such that
$$
Cf_1=e_1+a\delta_0,
\qquad
Cf_2=e_1+e_2-e_3+c\delta_0.
$$
The condition $C\delta_1=\delta_0$, together with
$\delta_1=-f_1-f_2+2f_3$, then gives
$$
Cf_3=e_3+\tfrac{1}{2}(a+c)\delta_0.
$$
Since $\delta_0$ is primitive, $a+c$ is even.  Writing $a+c=2b$, and then lifting
from $\Lambda^\times$ to $\Lambda$ in a manner respecting (\ref{filt1}), one obtains
$$
\begin{aligned}
Cf_1&=e_1+a\delta_0,\\
Cf_2&=e_1+e_2-e_3+(2b-a)\delta_0,\\
Cf_3&=e_3+b\delta_0,\\
Cf_4&=e_4+pe_1+qe_2+re_3
\end{aligned}
$$
for some $a,b\in \bZ$, with $p,q,r\in\bZ$ determining the further
lift from $\Lambda^\times$ to $\Lambda$.  Equivalently,
$$
C=\small{
\begin{pmatrix}
1-2a&1-4b+2a&-2b&p\\
-a&1-2b+a&-b&q\\
3a&-1+6b-3a&1+3b&r\\
0&0&0&1
\end{pmatrix}}.
$$

We now use the integral monodromy of the Mumford construction at
$\infty$ to determine these five integers.  Put
$$
N_\infty=(T_0CT_1C^{-1})^{-1}-I.
$$
The description in Proposition \ref{conj-rat} 
of the monodromy of the Mumford degeneration implies that
$N_\infty^2=0$ and that ${\rm im}(N_\infty)$ is a primitive rank $2$
sublattice of $\Lambda$.  Direct multiplication gives
$$
N_\infty^2e_j=a\delta_0\quad (j=1,2,3),
$$
so $a=0$.  After setting $a=0$, its remaining column is
$$
N_\infty^2e_4=(3q+r)
((1-4b)e_1-2be_2+(6b-1)e_3),
$$
and hence $r=-3q$.  Set $m=p-2q$.  The matrix of $N_\infty$ now becomes
$$
\small{
\begin{pmatrix}
1-4b&1-4b&1-4b&4bm\\
-2b&-2b&-2b&(2b-1)m\\
6b-1&6b-1&6b-1&(1-6b)m\\
0&0&0&0
\end{pmatrix}}.
$$
The greatest common divisor of its $2\times2$ minors is
$|m(1-6b)|$.  Primitivity of ${\rm im}(N_\infty)$ therefore implies
$b=0$ and $m=\pm1$.
Since $b=0$, 
$$
N_\infty e_j=e_1-e_3\quad (j=1,2,3),
\qquad N_\infty e_4=\pm(e_3-e_2).
$$
The first vanishing cycle comes from the $I_1$ monodromy of the
elliptic surface, while the second comes from the linearization.  For
a counterclockwise meridian, the condition
$\operatorname{ord}_\infty(\phi)=1$ gives $m=1$; a pole would give
$m=-1$ and negate precisely the second contribution to the monodromy
of the Mumford construction.  It follows that
$$
p=2q+1,\qquad r=-3q,
$$
and therefore
$$
Cf_4=e_4+e_1-q\delta_0.
$$
The remaining integer $q$ is precisely the freedom to replace 
$e_4\mapsto e_4+k\delta_0$. This shear commutes with the local monodromy
at $0$, as both $e_4$ and $\delta_0$ are fixed. 
Thus, we may adjust the marking
so that $q=0$.  Then
\begin{align*} 
&Cf_1=e_1, &&Cf_2=e_1+e_2-e_3, \\
 &Cf_3=e_3, 
 &&Cf_4=e_1+e_4,
\end{align*}
which proves the displayed formula for $C$.
\end{proof}

Since
$X(u)\to \bP^1$ has a section, the two local monodromy
matrices, together with the change of basis $C$,
 determine the topology of the real $4$-torus bundle
 arising from restricting to $\bP^1\setminus \{0,1,\infty\}$.

\section{The fundamental group}\label{sec-fd}

Put
$
U=\bP^1\setminus\{0,1,\infty\},$ and 
$J=Y\vert_U=X(u)\vert_U.
$
The section of Proposition~\ref{section-of-J} gives a base point in every
fiber.  Hence
$$
\pi_1(J)=\Lambda\rtimes\pi_1(U),
$$
where the action on $\Lambda$ is the monodromy representation.
Write
$
\mathcal B_0=(e_1,e_2,e_3,e_4)
$
for the basis of Theorem~\ref{matrices}, normalized so that
$$
\psi(e_4)=1,
\qquad
\psi(e_1)=\psi(e_2)=\psi(e_3)=0.
$$
We use the same letter $\psi$ for its extension to
$\Lambda\otimes\bQ$. In the common basis $\mathcal B_0$ we have
\begin{align}\label{common-basis}
[T_1]_{\mathcal B_0}=C[T_1]_{\mathcal B_1}C^{-1}
=\small{
\begin{pmatrix}
-1&1&-1&2\\
-1&0&-1&1\\
1&1&2&-1\\
0&0&0&1
\end{pmatrix}}.
\end{align}
In particular, $\psi$ is fixed by both monodromies.

Since $T_\infty=(T_0T_1)^{-1}$, the preceding matrices give
$$
(T_\infty-I)\Lambda
=\langle e_1-e_3,e_3-e_2\rangle=:K_\infty.
$$
These are precisely the two vanishing cycles (i.e.~generators of $W_{-2}H_1$)
 of the  Mumford construction at $\infty$.  Thus
$
K_\infty\subset \pi_1(J)$
is killed by filling the fiber at $\infty$ (these cycles collapse to points at either of the
two $0$-dimensional toric strata of the glued $dP_6$).  Since
$$
T_0(e_1-e_3)=e_1+e_2-e_3,
$$
the three vectors $e_1-e_3$, $e_3-e_2$, and $e_1+e_2-e_3$ form a basis
of $\ker\psi$.  Hence the normal closure of $K_\infty$ is exactly
$$
\ker\psi=\langle e_1,e_2,e_3\rangle.
$$
Thus the image of $\Lambda$ in $\pi_1$ of the extension of $J$ over $\infty$ is infinite cyclic,
generated by the image $c$ of $e_4$.  Because $\psi$ is monodromy invariant, $c$ is
central.

Put
$v_0=3a_0\in\Lambda^{T_0},$
$v_1=4a_1\in\Lambda^{T_1},$
and set
$$
m=\psi(v_0)=3\psi(a_0)\in \bZ,
\qquad
n=\psi(v_1)=4\psi(a_1)\in \bZ.
$$
Let $x_0,x_1$ be the meridians $\gamma_0, \gamma_1$ about $0,1$ lifted
to the section of $J$. Since the section extends
over the Mumford construction at $\infty$, we have the relation
$x_0x_1=1$.  

After quotienting by $\ker(\psi)$, 
the two logarithmic fillings over $0$, $1$ by the multiple fibers impose relations, see (\ref{pi1-ref}) below,
$$
x_0^3=c^m,
\qquad
x_1^4=c^n.
$$

\begin{proposition}\label{fundamental-group}
Let $p\coloneqq 12\psi(a_0+a_1)$.
The fundamental group of $Y$ is trivial when $|p|=1$. \end{proposition}

\begin{proof}
The preceding discussion gives
$$
\pi_1(Y)=
\left\langle c,x_0,x_1 \, \big{|} \,
c\text{ central},\,\,\, x_0x_1=1,\,\,\,
x_0^3=c^m,\,\,\, x_1^4=c^n
\right\rangle.
$$
Eliminating $x_1=x_0^{-1}$ makes the group abelian, generated by $x_0$ and
$c$, and the remaining relations reduce this abelian group to a cyclic group 
of order $|4m+3n|=|p|=1$. \end{proof}

\section{The integral homology}\label{sec-homology}

Let
$
f:Y\to\bP^1
$
be the torus fibration and put
$
V={\rm Hom}(\Lambda,\bZ).
$
Let $(e_1^*,e_2^*,e_3^*,e_4^*)$ be the basis of $V$ dual to
$\mathcal B_0=(e_1,e_2,e_3,e_4)$ (see Theorem \ref{matrices}).  The filtration
$
\bZ\delta_0\subset\Lambda^\times\subset\Lambda
$
induces the dual filtration
$$
0\subset W_0\subset W_1\subset W_2=V,
\qquad
W_0=(\Lambda^\times)^\perp,
\qquad
W_1=(\bZ\delta_0)^\perp.
$$
Thus
$
W_0=\bZ e_4^*=\bZ\psi
$
and, since $\delta_0=-2e_1-e_2+3e_3$,
$$
W_1=
\left\langle e_1^*-2e_2^*,\ 3e_2^*+e_3^*,\ e_4^*\right\rangle.
$$
Over $U$, the local system $R^kf_*\bZ$ has fiber
$\wedge^kV$.  Recall that the fibers over $0$, $1$ are multiple
with bielliptic reductions, while the fiber at
$\infty$ is the Mumford construction central fiber of
Section~\ref{sec-initial}.

\begin{proposition}\label{euler-characteristic}
$\chi_{\rm top}(Y)=2$.
\end{proposition}

\begin{proof}
Since a complex $2$-torus and both reduced bielliptic fibers 
have Euler characteristic zero, additivity of the Euler characteristic reduces 
the computation to the central fiber $Y_\infty$ of the Mumford construction.  Its toric
stratification has exactly two zero-dimensional strata.  
Positive-dimensional strata contain a factor $\bC^*$ and hence have
Euler characteristic zero. So
$
\chi_{\rm top}(Y)=\chi_{\rm top}(Y_\infty)=2.
$
\end{proof}

\begin{lemma}\label{invariant-form}
The monodromy preserves the primitive alternating form
$$
\xi=3e_1^*\wedge e_2^*+e_1^*\wedge e_3^*
-2e_2^*\wedge e_3^*-2e_3^*\wedge e_4^*.
$$
Its elementary divisors are $(1,6)$.
\end{lemma}

\begin{proof}
Taking inverse transposes of the matrices of Theorem~\ref{matrices}
shows directly that $T_i^*\xi=\xi$ for $i=0,1,\infty$. Computing 
$$\tfrac{1}{2}\xi \wedge \xi=-6e_1^*\wedge e_2^* \wedge e_3^* \wedge e_4^*$$
proves that the elementary divisors are $(1,6)$.
\end{proof}

\begin{lemma}\label{open-invariants}
The monodromy-invariant classes over $U$ are
$$
\begin{aligned}
H^0(U,R^1f_*\bZ)&=\bZ \psi &&(= W_0),\\
H^0(U,R^2f_*\bZ)&=\bZ\xi, \\
H^0(U,R^3f_*\bZ)&=\bZ\theta && (= \wedge^3W_1),
\end{aligned}
$$
where
$
\theta\coloneqq\xi\wedge \psi$.
\end{lemma}

\begin{proof}
The first equality is Lemma \ref{covect}. The invariance
of $\xi$ is Lemma \ref{invariant-form}, and the invariance
of $\theta = \xi\wedge\psi$ follows. Note that $\wedge^3W_1$ is also
monodromy invariant by general principle. Performing a linear algebra check with
the matrices  $[T_0]_{\cB_0}$ and $[T_1]_{\cB_0}$ of Theorem \ref{matrices}
and (\ref{common-basis}), one can also directly compute
that the classes $\psi$, $\xi$, $\theta$ integrally generate the 
monodromy invariants in $\wedge^kV$,
for $k=1,2,3$, respectively.
\end{proof}

We next determine what multiples of these invariant classes extend over the three
singular fibers.

\begin{lemma}\label{specialization-lattices}
Suppose $|p|=1$.  For $i=0,1$, let
$\pi_i:F_i\to G_i$ be the unramified cover from the smooth good reduction
fiber to the reduced bielliptic fiber.  Write
$$
{\rm sp}_i^k\coloneqq\pi_i^*:H^k(G_i,\bZ)\to
H^k(F_i,\bZ)^{T_i}=(\wedge^kV)^{T_i}.
$$
Then
\begin{align}\label{table}
\begin{aligned}
[(\wedge^kV)^{T_i}:{\rm im}({\rm sp}_i^k)]
=
\begin{array}{c|ccc}
 i\backslash k&1&2&3\\ \hline
0&3&1&1\\
1&4&2&2
\end{array}
\end{aligned}\end{align}
with cyclic cokernel, generated by the images
of $\psi$, $\xi$, $\theta$ in respective degrees $k=1$, $2$, $3$.
At $\infty$, the specialization map is an isomorphism onto
$(\wedge^kV)^{T_\infty}$ for degrees $k=1$, $2$, $3$.
\end{lemma}

\begin{proof}
The local invariant lattices are
\begin{align}\label{some-bases}
\begin{aligned}
V^{T_0}&=\langle e_3^*,e_4^*\rangle,
& (\wedge^2V)^{T_0}&=\langle e_3^*\wedge e_4^*,\xi\rangle,\\
V^{T_1}&=\langle e_2^*+e_3^*,e_4^*\rangle,
& (\wedge^2V)^{T_1}
&=\langle(e_2^*+e_3^*)\wedge e_4^*,\xi\rangle.
\end{aligned}
\end{align}

Let $d_0=3$ and $d_1=4$. The affine
group $\pi_1(G_i)$ is generated by the translations $t_\lambda$ and by
a lift $\widetilde{\rho}_i$ 
of $\widetilde{\rho}$ (see Construction \ref{con-log}) to the universal cover $\bC^2\to F_i$. 
More precisely, choose $\rho_i$ to lift the order $d_i$ deck action on $L^\times$ over $X'_i$ and
set $\widetilde{\rho}_i=t_{a_i}\circ\rho_i$. We have
\begin{align}\label{pi1-ref}
\widetilde{\rho}_i t_\lambda\widetilde{\rho}_i^{\,-1}=t_{T_i^{-1}\lambda},
\qquad
\widetilde{\rho}_i^{\,d_i}=t_{v_i}\circ \rho_i^{d_i} = t_{v_i+w_i} 
\end{align}
where $v_i=d_ia_i$ and $w_i\in \Lambda^{T_i}$ is defined by
$\rho_i^{d_i}=t_{w_i}$.
Thus
$$
{\rm im}({\rm sp}_i^1)=
\{\alpha\in V^{T_i}:d_i\mid\alpha(v_i+w_i)\}.
$$
Since $|p|=1$, $e_4^*(v_0)=m$ is a unit modulo $3$ and
$e_4^*(v_1)=n$ is a unit modulo $4$.  We have $e_4^*(w_i)=\psi(w_i)=0$
because $w_i\in \pi_1(L^\times)=\ker(\psi)$. Evaluation on $v_i+w_i$
therefore gives an isomorphism
$$
V^{T_i}/{\rm im}({\rm sp}_i^1)\simeq\bZ/d_i\bZ,
$$
generated by the class of $e_4^*$. The first column of (\ref{table})
follows.

The indices in degree $2$ follow from computing the intersection pairings
on $H^2(F_i,\bZ)$ and $H^2(G_i,\bZ)$.  
The Gram matrices of the intersection pairing 
of the bases (\ref{some-bases}) of $H^2(F_i,\bZ)^{T_i}$ are
$$
\begin{pmatrix}0&3\\3&-12\end{pmatrix},
\qquad
\begin{pmatrix}0&2\\2&-12\end{pmatrix}.
$$
The Bagnera--de Franchis types of $G_i$ 
ensure that $H^2(G_i,\bZ)$ is free of rank $2$
\cite[Rem.~1.6]{Ser90}.
By Poincar\'e duality, the intersection pairing on its middle cohomology is
unimodular.  

Write $(\pi_i)_*$ for proper pushforward.  It satisfies
$(\pi_i)_*\pi_i^*=d_i$ and so every ${\rm sp}_i^k$ is
injective.  Since pullback multiplies the intersection 
pairing by $d_i$, comparison of determinants gives
$$
[(\wedge^2V)^{T_0}:{\rm im}({\rm sp}_0^2)]=1,
\qquad
[(\wedge^2V)^{T_1}:{\rm im}({\rm sp}_1^2)]=2.
$$
The latter cokernel is generated by $\xi$: if
$\xi$ descended at the order $4$ fiber, its square downstairs would be
$\xi^2/4=-3$, but the intersection form of $G_1$ is even (as $K_{G_1}\sim_\bQ 0$).
The second column of (\ref{table}) follows.

Finally, we make the
degree $3$ cohomology computation. The local invariant lattices in 
degree $3$ are
$$
\begin{aligned}
(\wedge^3V)^{T_0}
&=\left\langle e_1^*\wedge e_2^*\wedge e_3^*,\ \theta\right\rangle,\\
(\wedge^3V)^{T_1}
&=\left\langle
e_1^*\wedge e_2^*\wedge e_3^*
-e_2^*\wedge e_3^*\wedge e_4^*,\ \theta
\right\rangle.
\end{aligned}
$$
Pairing these bases with the respective bases of $V^{T_i}$ in
(\ref{some-bases}) gives
$$
\begin{pmatrix}0&3\\-1&0\end{pmatrix},
\qquad
\begin{pmatrix}0&2\\-1&0\end{pmatrix}.
$$
Let
$
b_i=[(\wedge^3V)^{T_i}:{\rm im}({\rm sp}_i^3)].
$
The pairing between $H^1(G_i,\bZ)$ and $H^3(G_i,\bZ)$ is unimodular,
whereas pullback to $F_i$ multiplies it by $d_i$.  Since both groups have rank $2$,
the determinant of the pulled-back pairing is therefore $d_i^2$.
On the other hand, replacing the two invariant lattices by sublattices of
indices $d_i$ and $b_i$ multiplies the determinant by $d_ib_i$.
The two displayed pairing matrices therefore give
$$
3\cdot b_0\cdot3=3^2,
\qquad
4\cdot b_1\cdot2=4^2.
$$
Thus $b_0=1$ and $b_1=2$.  Finally, the description of
${\rm im}({\rm sp}_1^1)$ allows us to choose $\ell\in\bZ$ so that
$e_2^*+e_3^*-\ell e_4^*$ descends to $G_1$.  Its pairing with $\theta$
is $2$.  If $\theta$ also descended, this pairing would be divisible by
$d_1=4$, a contradiction.  Hence the degree $3$ cokernel at $1$ is
generated by the class of $\theta$.  This proves the third column of
(\ref{table}).

At $\infty$, we have ${\rm im}\,N_\infty = \ker N_\infty = \ker(T_\infty-I)$, where $N_\infty:=T_\infty-I$.
Indeed, the $q$-order matrix
$\left(\begin{smallmatrix}1&0\\0&1\end{smallmatrix}\right)\in \GL_2(\bZ)$
of Proposition \ref{mum-cor} is unimodular.
The last sentence of the lemma follows for $k=1$. 
A computation of the Mayer--Vietoris spectral sequence for the toric stratification of 
$Y_\infty$ shows that the specialization maps to the monodromy invariant classes
are also isomorphisms for $k=2,3$.
\end{proof}

We now take cohomology on the base.

\begin{lemma}\label{pushforward-cohomology}
Suppose $|p|=1$.  Then
$$
\begin{aligned}
H^0(\bP^1,R^1f_*\bZ)&=12\bZ\psi,\\
H^0(\bP^1,R^2f_*\bZ)&=2\bZ\xi,\\
H^0(\bP^1,R^3f_*\bZ)&=2\bZ\theta,
\end{aligned}
$$
and
$
H^1(\bP^1,R^1f_*\bZ)=
H^1(\bP^1,R^2f_*\bZ)=0.
$
If $H^2(\bP^1,\bZ)=\bZ \omega$, then
$$
\begin{aligned}
H^2(\bP^1,R^0f_*\bZ)&=\bZ\omega,\\
H^2(\bP^1,R^1f_*\bZ)&=
\bZ\omega[e_1^*-e_2^*],\\
H^2(\bP^1,R^2f_*\bZ)&=
\bZ\omega[(e_1^*-e_2^*)\wedge\psi].
\end{aligned}
$$
\end{lemma}

\begin{proof}
Intersecting the three specialization lattices of
Lemma~\ref{specialization-lattices} inside the invariant lattices of
Lemma~\ref{open-invariants} gives the asserted $H^0$ groups.

For
$j:U\hookrightarrow\bP^1$ there is an exact sequence
\begin{align}\begin{aligned}\label{exact-01}
&\,\,\,0\to R^kf_*\bZ\to j_*(\wedge^kV)\to Q^k\to0, \textrm{ where } \\
Q^1&=(\bZ/3\bZ)_0\oplus (\bZ/4\bZ)_1 \textrm{ and }
Q^2=Q^3=(\bZ/2\bZ)_1.
\end{aligned}\end{align}
Local-coefficient duality on the pair of pants gives
$$
H^2(\bP^1,j_*L)=L_{\langle T_0,T_1\rangle}
$$
where the subscript denotes coinvariants.
Row reduction gives
$$
\begin{aligned}
V_{\langle T_0,T_1\rangle}
&=\bZ[e_1^*-e_2^*],\\
(\wedge^2V)_{\langle T_0,T_1\rangle}
&=\bZ[(e_1^*-e_2^*)\wedge\psi],\\
[\xi]&=12[(e_1^*-e_2^*)\wedge\psi].
\end{aligned}
$$
Since the skyscraper sheaf $Q^k$ has no higher cohomology, this proves
the assertions about $H^2$.

Finally we check the vanishing of the $H^1$ groups.
We have a surjection $H^0(\bP^1, j_*(\wedge^kV))\to H^0(\bP^1, Q^k)$, for 
$k=1,2$---the classes $\psi$, $\xi$ map to generators of $H^0(\bP^1, Q^1) = \bZ/3\bZ\oplus \bZ/4\bZ$,
$H^0(\bP^1, Q^2) = \bZ/2\bZ$, respectively, by Lemma \ref{specialization-lattices}. 
From the long exact sequence
of (\ref{exact-01}), it suffices to prove vanishing of 
$H^1(\bP^1, j_*(\wedge^kV))$.
 For
$L=V$ and $\wedge^2V$, respectively, this group is identified with
$$
\frac{(T_0-I)L\cap(T_\infty-I)L}{(T_0-I)(L^{T_1})}
$$
if we normalize cocycles to vanish on $\gamma_1\in \pi_1(U)$.
For $L=V$ the numerator and denominator are
$\bZ(e_1^*+e_2^*+e_3^*)$; for $L=\wedge^2V$ they are
$\bZ((e_1^*+e_2^*+e_3^*)\wedge\psi)$.  Thus both groups
vanish.
\end{proof}

We now compute the differentials $d_2$ on the $E_2$-page of the Leray
spectral sequence that contribute in total degree at most $3$.

\begin{lemma}\label{transgression-calculation}
We have
$$
d_2(12\psi)=p\omega,\quad
d_2(2\xi)=-p\omega[e_1^*-e_2^*],\quad
d_2(2\theta)=p\omega[(e_1^*-e_2^*)\wedge\psi].
$$
\end{lemma}

\begin{proof}
The first differential measures the incompatibility of local lifts of
$12\psi$.  The relations
$$
\widetilde{\rho}_0^{\,3}=t_{v_0+w_0},
\qquad
\widetilde{\rho}_1^{\,4}=t_{v_1+w_1}
$$
show that the contributions from the two finite fibers are
$$
\frac{12\psi(v_0)}{3}=4m,
\qquad
\frac{12\psi(v_1)}{4}=3n.
$$
The fiber at $\infty$ contributes zero as $\psi$ extends. 
Hence, after fixing the common orientation,
$d_2(12\psi)=p\omega$.
Write
$$
d_2(2\xi)=-p'\omega[e_1^*-e_2^*],
\qquad
d_2(2\theta)=p''\omega[(e_1^*-e_2^*)\wedge\psi].
$$
The Leray spectral sequence is multiplicative and $d_2$ is a derivation.
Hence, applying $d_2$ to
$$
(2\xi)(12\psi)=12(2\theta)
$$
and using the identities
$$
\xi\wedge\psi=\theta,
\qquad
[\xi]=12[(e_1^*-e_2^*)\wedge\psi]
$$
gives $p''=2p-p'$.  If
$\nu=e_1^*\wedge e_2^*\wedge e_3^*\wedge e_4^*$, then
$$
(e_1^*-e_2^*)\wedge\theta=-\nu,
\qquad
\xi\wedge(e_1^*-e_2^*)\wedge\psi=-\nu.
$$
Applying $d_2$ to $(2\xi)(2\theta)=0$ therefore gives
$2(p'-p'')\omega\nu=0$. The same coinvariant argument 
as in Lemma \ref{pushforward-cohomology}
gives $H^2(\bP^1,R^4f_*\bZ)=\bZ\omega[\nu]$. 
Hence $p'=p''$, and consequently
$p'=p''=p$.
\end{proof}

\begin{theorem}\label{integral-homology}
If
$|p|=
\left|12\psi(a_0+a_1)\right|=1,
$
then
$$
H^k(Y,\bZ)=
\begin{cases}
\bZ,&k=0,6,\\
0,&1\leq k\leq5.
\end{cases}
$$
Thus $Y$ is an integral homology $6$-sphere.
\end{theorem}

\begin{proof}
The only differentials on the $E_2$-page relevant in total degrees at most $3$ are the three
maps of Lemma~\ref{transgression-calculation}.  When $|p|=1$, all three are
isomorphisms.  Lemma~\ref{pushforward-cohomology} then gives directly
$$
H^1(Y,\bZ)=H^2(Y,\bZ)=H^3(Y,\bZ)=0.
$$
Since $Y$
is closed and oriented of real dimension $6$, Poincar\'e duality and
universal coefficient theorem determine the remaining groups.
\end{proof}

\section{Additional remarks}\label{sec-speculate}

\begin{remark} Other than the discrete parameters in Theorem \ref{main-thm},
there is the continuous parameter $u\in \Delta^*$. We may thus build
a proper, holomorphic submersion $\cY^*\to \Delta^*$ whose fiber over $u$ is $X(u,a_0,a_1)$,
and similarly, for the untwisted version, a proper holomorphic submersion $\cX^*\to \Delta^*$ whose
fiber over $u$ is $X(u)$.

The complex structure on $X(u)$ genuinely varies with $u$.
Its algebraic dimension is $1$, so the map to $\bP^1$ is unique
after requiring that the singular fibers lie over $0,1,\infty$. Then,
the periods of the fiber over $t\in \bP^1\setminus \{0,1,\infty\}$ vary.
Alternatively, the second elliptic curves $E_0'$ and $E_1'$ involved in the bielliptic fibers
(Corollary \ref{fill}) vary depending on the action of $\bZ$ on $L^\times$
and hence depend on the scaling parameter $u\phi_0$ of the linearization,
see Definition \ref{linearization-def}.

Does this family of complex manifolds degenerate in a nice way over the origin
$0\in \Delta$? The geometric construction
is suggestive:~The parameter $u$ is a rescaling parameter for the linearization 
$\phi\colon t_{6P}^*M\to M$ discussed in Section \ref{sec-initial}. Thus as we take the limit $|u|\to 0$, 
the modulus of the annular fundamental domain for multiplication by $u$ on $\bG_m$ 
grows to infinity. Since, furthermore, the linearization $\phi$ lifts the translation $t_{6P}^{-1}$, this
strongly suggests that the limit $$X(0)\coloneqq \lim_{u\to 0} X(u)$$ can be constructed
as follows:~Take the $\bP^1$-bundle $\bP_S(\cO\oplus M)$ and glue the zero section to the infinity
section by $t_{6P}$. This can thus be viewed fiberwise over $t\in \bP^1$ as a ``compactified semiabelian''
degeneration of the corresponding family of complex $2$-tori $X_t(u)$ as $u\to0$.

\begin{conjecture} 
There is a flat, proper filling $\cX\to \Delta$ and the central fiber $X(0)$ is non-normal. The normalization
is $\bP_S(\cO\oplus M)$ and the normalization map glues the zero and infinity sections via $t_{6P}$.
\end{conjecture}

The family $\cY^*\to \Delta^*$ is particularly appealing as it is a smooth $\bS^6$-bundle
over the punctured disk. It is less clear what the limit should be, but a natural guess 
is a logarithmic transform of $X(0)$, whose fibers over $0$ and $1$ have multiplicity $3$ 
and $4$, and whose reductions are ``semiabelian type'' degenerations of bielliptic surfaces.
\end{remark}

\begin{remark} The rank $4$ local system of first homology groups of the fibers of $Y\to \bP^1$ contains
a natural rank $3$ sub-local system, given by the kernel of the unique invariant (up to scale) 
covector $\psi$ of Lemma \ref{covect}. Geometrically, this rank $3$ sub-local system corresponds
to a real $3$-dimensional subtorus of the fibers.
Quotienting fiberwise by the translational action of this subtorus, we reduce the fibers $Y_t$
to circles, but some care must be taken at $t=0,1,\infty\in \bP^1$.

At $t=\infty$, the kernel of $\psi$
contains the vanishing cycles $W_{-2}H_1$ of the nearby fiber. The quotient by the torus
spanned by vanishing
cycles can be interpreted in the following manner:~Take the Clemens collapse \cite{Cle77} from
the nearby fiber to $Y_\infty$ and then compose with the moment mapping on the glued $dP_6$
to a glued hexagon (which is a polytope of this Mumford construction). 

There is, in the construction, a distinguished edge of the glued hexagon. Indeed, 
${\rm Aut}^0(Y)\simeq \bC^*$ (acting by scaling on the prequotient 
$M^\times$ as a $\bG_m$-torsor) and the fixed locus of this $\bC^*$-action is one
of the three glued edges. This distinguished edge determines a further quotient
of the hexagonal torus to a circle. Thus over $\infty$, this quotient map of the first paragraph
extends continuously to the composition of
the moment mapping with a further quotient to a circle.

At $0$ or $1$, the situation is different. Here, the rank $3$ sub-local system contains
the local rank $2$ sub-local system on which the monodromy about $0$ or $1$ is nontrivial.
Quotienting by the corresponding subtorus gives a multiple elliptic curve
isogenous to $E'$
because the twisted action defining the logarithmic transform producing the bielliptic multiple fiber 
acts by a torsion translation on $E'$.
The further quotient collapses this multiple elliptic curve to
a circle, with multiplicity $3$ or $4$.

The conclusion is that the quotient by the torus associated to this natural rank $3$ sub-local
system is a map $\bS^6\to Q$, where $Q$ is a Seifert fibered $3$-manifold over $\bP^1\simeq \bS^2$ 
with a multiple circle of multiplicity $3$ over $0$ and multiplicity $4$ over $1$. The condition 
$|12\psi(a_0+a_1)|=1$ in Theorem \ref{main-thm} translates exactly into the condition
$Q\simeq \bS^3$ with $\bS^3\to \bS^2$ the Seifert fibration with multiplicities $3$ and $4$ over
two points \cite[Ch.~1]{Orl72}. Thus the quotient map is a map $\bS^6\to \bS^3$.
\end{remark}

\begin{remark} A very general smooth fiber $F$ of $X(u)\to\bP^1$, or equivalently $Y\to\bP^1$, 
is not an abelian variety, but one can check \cite[\S 9.1]{AC26} that it admits a monodromy-invariant
 Hodge class $$\xi\in H^{1,1}(F)\cap H^2(F,\bZ).$$ The class $\xi$ (see Lemma \ref{invariant-form})
 is thus the first Chern class of a line bundle. Its type, as a symplectic form, is $(1,6)$. On $F^0H_1(F,\bC)$, the
 associated Hodge form $i\xi(v,\bar v)$ has signature $(1,1)$. Thus we have a natural
 period mapping $$\bP^1\setminus \{0,1,\infty\}\to {\rm Sp}(\bZ^4,\xi)\backslash {\rm Sp}_4(\bR)/{\rm U}(1,1).$$
 The difference from the abelian case is that the action of ${\rm Sp}(\bZ^4,\xi)$ on the domain of 
 complex tori $\bD\coloneqq {\rm Sp}_4(\bR)/{\rm U}(1,1)$ is not properly discontinuous. But, after all, the monodromy
 is not all of ${\rm Sp}(\bZ^4,\xi)$---indeed, the monodromy preserves the filtration (\ref{filt1}, \ref{filt2}) and so lies
 in a parabolic subgroup $P_\bZ\subset \Sp(\bZ^4,\xi)$.
 B.~Klingler suggested the possibility that the monodromy group
 could act properly discontinuously on $\bD$, or at least on an analytic open set
 containing the period image.
\end{remark}

\begin{remark} The constructions of Sections \ref{sec-initial} and \ref{sec-log-mod} work in significantly greater
generality than this one example. Indeed, consider an elliptic surface $\pi\colon S\to C$ with section, a line bundle
$M\in {\rm Pic}(S)$ which lies in ${\rm Pic}^0$ of the general fiber, and a translation $t\in {\rm Aut}(S/C)$.
Suppose $\cH om(t^*M,M)\simeq \pi^*\cO_C(\pm E)$ for $E$ an effective divisor 
(e.g.~this holds when $C\simeq \bP^1$ and $t$ preserves the components of the Kodaira fibers).
We get an associated linearization $\phi\in {\rm Aut}(M^{\times,\circ})$ of the principal
$\bG_m$-bundle $M^\times$ over the complement of the fibers over $E$.

Then $\phi$ has either zeros or poles along the inverse image of $E$, according to the sign.
In the former case, scale $\phi$ by a sufficiently small $u\in\bC^*$; in the latter,
by a sufficiently large one. This action is properly discontinuous by the argument 
of Proposition \ref{discontinuous} (since the linearization 
has only zeroes or only poles, some scaling is uniformly contracting
or expanding). Furthermore, we can fill the quotient by a Mumford construction, 
as in Proposition \ref{mumford-1}. One may then perform log transforms, at some
finite number of fibers.

It is unclear what class of $6$-manifolds can arise from such a construction.
This line of inquiry is pursued in {\it Threefold Explorer}, an applet powered
by SageMath. It automates 
van Kampen and Mayer--Vietoris computations of the fundamental group and integral
cohomology of a broad class of such constructions \cite{Eng26}.
\end{remark}

\end{document}